\documentclass{article} 
\usepackage{graphicx}
\usepackage{epstopdf}
\usepackage{mathptmx}   
\usepackage{amsmath,amsthm}
\usepackage{amssymb}
\usepackage{mathrsfs}
\newenvironment{acknowledgements}{{\bf Acknowledgments\ }}{}{}
\newtheorem{thm}{Theorem}
\newtheorem{theorem}[thm]{Theorem}

\newtheorem{corollary}[thm]{Corollary}

\newtheorem{lemma}[thm]{Lemma}
\theoremstyle{definition}

\renewcommand{\phi}{\varphi}
\renewcommand{\rho}{\varrho}
\newcommand{\C}{\mathbb C}
\newcommand{\R}{\mathbb R}

\newcommand{\N}{\mathbb N}
\newcommand{\LP}{\mathrm L}
\newcommand{\sS}{\mathscr S}

\newcommand {\schwartz}{\sS}

\newcommand{\ov}{\overline}                           
\newcommand{\menge}[2]{\left\{#1\ : \ #2\right\}} 
 
\newcommand{\norm}[1]{\left\lVert#1\right\rVert}     
\newcommand{\betrag}[1]{\left\lvert#1\right\rvert}    
\newcommand {\spd}[2]{\langle#1\vert#2\rangle}    

\newcommand {\dil}{D}   
\newcommand{\famdil }{\{\dil_t\}_t}
\newcommand{\me}{{-1}}
\newcommand{\halb}{\ensuremath{{\scriptstyle\frac{1}{2}}}}
\newcommand {\GL}{\mathrm{GL}}          
     
\newcommand {\mal}{\cdot}              
\newcommand {\krng}{\circ}             
\newcommand {\falt}{\ast}
\newcommand {\jap}[1]{ \left< #1 \right> }

\newcommand{\dht}{\,\frac{\mathrm dt}{t}}
\newcommand{\folgt}{\Rightarrow}
\newcommand {\gdw} {\ \Leftrightarrow\ }

\newcommand {\dd} {=}         

\newcommand{\leb}{{\mu}}        
\newcommand{\diff}{\ \mathrm d}
\newcommand{\bary}{\bar y}
\newcommand{\kq}{,\quad}

\newcommand{\inv}{^{-1}}
\DeclareMathOperator{\supp}{supp} 
 
\DeclareMathOperator{\id}{Id} 

\DeclareMathOperator{\im}{Im} 
\DeclareMathOperator{\aut}{Aut}  
\DeclareMathOperator{\trace}{trace}
\DeclareMathOperator{\deter}{det}

\newcommand{\buu}{\tilde B}
\newcommand{\kon}{k}
\newcommand{\laplace}{\Delta}
\newcommand{\ort}[1]{O_{#1}}
\newcommand{\ortp}{O_{t}}
\newcommand{\ortn}{O_{-t}}
\newcommand{\ortps}{O_{s}}

\begin{document}
\title{Homogeneous spaces adapted to singular integral operators involving rotations}
\author{H. F. Bloch}
\date{}
\maketitle
\begin{abstract}
Calder\'on-Zygmund decompositions of functions have been used to prove weak-type (1,1) boundedness of singular integral operators. In many examples, the decomposition is done with respect to a family of balls that corresponds to some family of dilations. 
We study singular integral operators $T$ that require more particular families of balls,  providing new spaces of homogeneous  type.
Rotations play a decisive role in the construction of these balls.
Boundedness of $T$ can then be shown via Calder\'on-Zygmund decompositions with respect to these spaces of homogeneous type.
We prove weak-type (1,1) and $\LP^p$ estimates for operators $T$ acting on $\LP^p(G)$, where $G$ is a homogeneous Lie group.
Our results apply to the setting where the underlying group is the Heisenberg group and the rotations are symplectic automorphisms. They also apply to operators that arise from some hydrodynamical problem involving rotations.
\end{abstract}

\section{Introduction and main result} \label{intro}
{\em An example where G is abelian. }
In their paper  on a "singular 'winding' integral operator",
Farwig, Hishida and M\"uller \cite{farwig} considered the following problem.
A rigid body rotating with  fixed angular velocity is surrounded by an incompressible viscous fluid. This situation can be described by  the Navier-Stokes equations and  the assumption that the velocity $u$ of the fluid vanishes at infinity
and is equal to the local velocity of the body at its surface. 
While estimating $\laplace u$, an integral operator came into play. We consider a similar operator in a setting where
$G=\R^2$. Then $\mathrm d\leb=\mathrm dx$ is a Haar measure on $G$. We denote
 rotations in the plane  by
$ \ort{t}(x)\dd \left(
\begin{matrix}
\cos t & -\sin t\\
\sin t &\cos t
\end{matrix}
\right)\cdot x$, where $t\in\R$. A family of dilations is defined by $\dil_r(x)\dd r^{\frac{1}{2}}\cdot x$, where $r>0$.
Let $E(x)=e^{-\betrag{x}^2}$ and $E_r(x)=\frac{1}{r} E\krng \dil_{r^\me}(x)=\frac{1}{r}E(\frac{x}{\sqrt r}), \ x\in\R^2$.
\begin{figure*}
\includegraphics[width=0.6\textwidth]{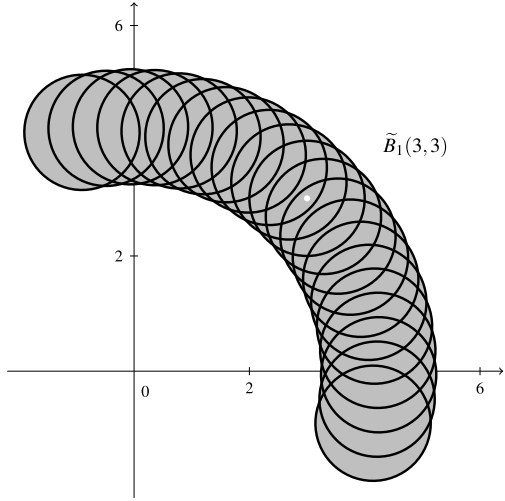}
\caption{The quasi-ball $\buu$ of radius $1$ and center $(3,3)$.}
\label{fig:1}      
\end{figure*}
We introduce a linear operator by setting
\begin{equation*}
 (Tf)(x) \dd \int_0^\infty (\laplace E_t )\ast(f\krng \ortp)(x) \,\mathrm dt,
 \end{equation*}
where $\laplace=\partial_1^2+\partial_2^2$, $f\in\schwartz(\R^2)$ and 
$ (f\ast g)(x)\dd\int_{\R^2} f(y)g(x-y)\,\mathrm dy$.
 Setting  $\psi=\laplace E$ and $\psi_r\dd \frac{1}{r}\psi\krng \dil_{r^\me}$, we obtain
\begin{equation}\label{TR2}
 (Tf)(x) = \int_0^\infty \psi_t\ast(f\krng \ortp)(x) \dht\kq f\in\schwartz,\ x\in\R^2.
 \end{equation}
Note that $\int \psi_r(x)\,\mathrm dx=0$, which implies some cancellation when smooth functions are convoluted with $\psi_t$ and $t$ is small. This cancellations suffices to make the integral in \eqref{TR2} converge.
The arguments used in \cite{farwig} 
suggest that  $T$ can be extended to a bounded linear operator on $\mathrm L^p$ if $1<p<\infty$. The operator $T$ is said to be of {\em weak type (1,1)} or bounded  from $\LP^1$ to $\LP^{1,\infty}$, if  $\leb\left(\menge{x}{\betrag{Tf(x)}>\lambda}\right)\le C\frac{\norm{f}_1}{\lambda}$ for all $\lambda>0$ and all $f$. 
So far, we had no hints on whether
$T$ is of weak type (1,1). We will show that this is indeed the case. The proof will rely on quasi-balls $\buu$ as the one depicted in
Figure \ref{fig:1}.
More precisely, we define the quasi-balls  $\buu_r$ as unions of Euclidean balls of radius $r^{\frac{1}{2}}$ by
\[  \buu_r(z)\dd  \bigcup_{-r\le s\le r}\menge{x\in\R^2}{\norm{x}_2<r^\frac{1}{2}} + \ort{s}(z),\quad z\in\R^2, r>0.  \]
{\em An example where $G$ is nonabelian. }
The set $\C^2\times\R$ together with the product 
\[ [(u_1,v_1,s_1),(u_2,v_2,s_2)]=(0,0,\im(\overline{u_1} u_2+\overline{v_1} v_2)) \]
 is a real Lie algebra. Furthermore, the product
\[ x\cdot y =x+y+\frac{1}{2}[x,y]\kq x,\, y \in \C^2\times\R \]
defines a real Lie group, the Heisenberg group $\mathbb H_2$. If $\alpha,\beta\in\R$, the symplectic automorphisms
$\ortp(u,v,s)=(ue^{i\alpha t}, ve^{i\beta t}, s)$ can be viewed as rotations and the automorphisms $\dil_t(u,v,s)=(tu,tv,t^2s)$  as nonisotropic dilations, yielding a noncommutative example for the general case.

The  more general setting presented below has been described and explained by Folland and Stein \cite[Chap. 1]{fost}. We add the notion of rotations. For the sake of simplicity, the assumptions are slightly redundant.

\paragraph{Assumptions}
Let $G$ be a connected and simply connected nilpotent Lie group with Lie algebra $\mathfrak g$, $\dim \mathfrak g=n\ge2$. We identify $G$ with
$\mathfrak g$ via the group exponential function, so $G$ is the manifold $\mathfrak g$ together with a
group product that is given by the Campbell-Hausdorff formula and the Lie product $[\cdot,\cdot]$ on $\mathfrak g$.
In this setting the neutral element of $G$ equals $0_{\mathfrak g}$. We have the identities $x^{-1}=-x$, $xy=x+y+\frac{1}{2}[x,y]+\ldots$, $\exp_G=\id$ and $\aut G=\aut \mathfrak g\subseteq\GL(\mathfrak g)$.

We assume that  $A$ is a diagonalizable linear operator on
$\mathfrak g$ whose eigenvalues are positive and that $\{\exp( A\log t)\}_t$ is a family of Lie algebra automorphisms. 
Furthermore, we assume  that $\{\dil_t\}_{ t>0}$ is a family of Lie group automorphisms such that $d\dil_t(e)=\exp( A\log t)$.
Then $\{\dil_t\}_{ t>0}$ is called a family of {\em dilations} and $G$ is called a {\em homogeneous group} with respect to $\{\dil_t\}_t$. 
$Q\dd \trace A$ is called its {\em homogeneous dimension}.
The maps $\dil_t=\exp( A\log t)$ are linear.

Let $O:\R \to \aut(G)$ be a continuous homomorphism whose image is relatively compact.  We also write $\{O_t\}_{t\in\R}$ instead of $O$ and refer to
$\{O_t\}_{t\in\R}$ as the family of {\em rotations}.
We assume that rotations commute with dilations, which amounts to saying that the eigenspaces of  nontrivial dilations are left invariant
by rotations.

From now on we identify $\mathfrak g$ (and hence $G$) with $\R^n$. 
The map $t\mapsto \deter \ortp$ is a continuous homomorphism from $(\R,+)$ to $(\R\setminus\{0\},\cdot)$. 
Since $\deter{\ort{\R}}$ is a subset of the compact set $\deter({\overline{\ort{\R}}})$ in $\R$, it is also bounded. Hence for every $t$ we have $\deter \ortp=1$ and $\ortp\in\mathrm{SL}(n)$.
Since the closure of $\menge{\ortp}{t\in\R}$  is a compact Lie subgroup of the connected group $\mathrm{SL}(n)$, it can be conjugated into the maximal compact subgroup $\mathrm{SO}(n)$ of $\mathrm{SL}(n)$ \cite{HN}. 
Therefore, the identification of $\mathfrak g$ with $\R^n$ can be done in such a way that $\ortp\in\mathrm{SO}(n)$ for all $t$, and we will proceed on this assumption.
We define $\schwartz(G)$ to be $\schwartz(\R^n)$ and
 use the Euclidean Schwartz norms
\[ \norm{f}_{(N)}\dd \sup_{\betrag{\alpha}\le N,\ x\in G} \jap{x}^N\betrag{\partial^\alpha f(x)} 
\kq  \jap{x}\dd \left( 1+\norm{x}_2^2\right)^{\frac{1}{2}}, \ x\in\R^n.  \]
Instead, Schwartz norms defined in terms of invariant vector fields and a homogeneous norm could be used, see \cite[p. 35]{fost}.

We use the Lebesgue measure $\mathrm dx$ on $G$ and denote it by $\mu$ when measuring sets, that is, $\mu(M)=\int_M\mathrm dx$. This is a left- and right-invariant Haar measure.
 Let  $\mathrm L^p$ and $\mathrm L^{1,\infty}$ be the Lorentz spaces $\mathrm L^p(G,\mathrm dx)$ and $\mathrm L^{1,\infty}(G,\mathrm dx)$ respectively. Note that $\norm{\,\cdot\,}_{1,\infty}$ is a quasinorm, and that $\mathrm L^{1,\infty}$ is complete \cite[Thm. 1.4.11.]{graf}. 
We denote the convolution of $f$ and $g$ by
$ (f\ast g)(x)= \int_G f(xy^{-1})g(y)\,\mathrm dy$.
For any function $\psi:G\to\C$ and $t>0$ we denote by $ \psi_t\dd t^{-Q}\psi\circ\dil_{t^{-1}} $
the $\mathrm L^1$-invariantly dilated function $\psi_t$.
\paragraph{The singular integral operator}
 We fix  some $\psi\in\schwartz(G)$ with $\int \psi(x)\,\mathrm dx=0$ and define
\begin{equation} \label{deftf}
 (Tf)(x) \dd  \int_0^\infty [\psi_t\falt(f\circ \ortp) ](x)\dht \kq f\in\schwartz ,\ x\in G.
\end{equation}
With techniques similar to the ones used in the proof of Lemma \ref{loesch}, it can be shown that \eqref{deftf} converges. For the proof of the following theorem, we will rely on a definition of $T$ as an operator on $\LP^2(G)$, which is compatible with \eqref{deftf}.
\begin{theorem}\label{mainthm}
Let  $1<p\le2$ and $T$ as in \eqref{deftf}.
There exists a constant $C$ such that for all $f\in\schwartz$ and all $\lambda>0$
we have the estimates
\begin{equation}\label{weaktype} 
\leb\left(\menge{x\in
G}{\betrag{Tf(x)}>\lambda}\right)\le C\frac{\norm{f}_1}{\lambda} 
\end{equation}
and
\begin{equation}  \label{lpestimate}
\norm{Tf}_p\le C \norm{f}_p.
\end{equation}
Hence $T$ has a unique extension to a bounded linear operator on $\LP^p$, $p\in]1,2]$ and a unique extension to a bounded linear
operator from $\mathrm L^1$ to $\mathrm L^{1,\infty}$.
On the premise that $\psi(O_t x)=\psi(x)$ for all $t$ and $x$, the preceding is also true for $2<p<\infty$. 
\end{theorem}

\section{Prerequisites}
{\em Homogeneous norms.\ } A continuous function $\betrag{\cdot}:G\to [0,\infty)$ is said  to be a 
homogeneous norm on $G$ with respect to $\famdil$ if it satisfies
$\betrag{x}=0\gdw x=0$,\  $\betrag{x^\me}=\betrag{x}$ and
\ $\betrag{\dil_tx}= t\betrag{x}$ for all $t>0$ and $x\in G$. Some authors require homogeneous norms to be smooth away from the origin.  
 There exists a
homogeneous norm $\betrag{\cdot}$ that is invariant under rotations; that is, 
\begin{equation} \label{rotinvariance} \betrag{\ort{t}x}=\betrag{x}
\end{equation}
 for all $x$, $t$. For a hint on how to produce such homogeneous norms, see the example in Sect. \ref{sec:example}.
 We will keep one such homogeneous norm fixed. 
The terms $\betrag{x\inv y}$ and $\betrag{xy^\me}$ define 
left-invariant and right-invariant quasi-distance functions respectively. These quasi-distances are symmetric and coincide if $G$ is abelian.  We refer to them as the {\em homogeneous distance} between $x$ and $y$. We use the term {\em quasi-metric} for symmetric quasi-distance functions; namely, if $d$ is a quasi-metric then $x=y\gdw d(x,y)=0$, \ $d(x,y)=d(y,x)$, and there is a constant $\kappa$ such that \ $d(x,y)\le\kappa(d(x,z)+d(z,y))$.  

{\em Balls and spheres. \ }
Balls and spheres with center $0$ and with respect to $\betrag{\cdot}$ are defined by $ B_r\dd \menge{x\in G}{\betrag{x}<r}$ and $ S_r\dd \menge{x\in G}{\betrag{x}=r}$ respectively. 
The ball with center $z$ and with respect to the
left- and right-invariant quasi-metrics are equal to
$z \mal B_r$ and $B_r\mal z$ respectively.

{\em Integration. \ }
The definition of $Q$ is natural in the sense that there exists a constant $C$ so that $ \int_{B_r\cdot z} \,\mathrm dx = Cr^Q,\ \text{for all }z\in G,\ r>0$ and that $\det D_r=r^Q$. Furthermore, there exists a positive Borel measure $\sigma$ on $S_1$ such that all $f\in\mathrm L^1(G)$ can be integrated using {\em  spherical coordinates} by
\begin{equation}\label{spherical}
 \int_G f(x) \,dx= \int_0^\infty\int_{S_1} f(D_r x) \,\mathrm d\sigma(x)r^{Q-1}\,\mathrm dr. 
\end{equation}
Since $\ortp\in\mathrm{SL}(n)$, $\mu$ is invariant under rotations:
\begin{equation}\label{rotinv}
\int f(x)\,\mathrm dx=\int( f\circ\ortp)(x)\,\mathrm dx
\end{equation}
for every $f$ and $t$. 
Note that $  f_r \ast g_r = (f\ast g)_r$, and that $(f\falt g)\krng \ortp =
(f\krng\ortp)\ast (g\krng\ortp)$, for all $f,g\in \LP^1(G)$, $t\in\R$, $s>0$.

{\em Constants. \ } We will use miscellaneous constants $C, C', C_1>0$ etc. whose values vary from line to line and who may depend on the geometric setting, for example, on the homogeneous group. Occasionally we write $a(x)\lesssim b(x)$ or  $a\lesssim b$ to indicate that there
is a constant $C$ such that $a(x)\le C\cdot b(x)$ for all $x$. Furthermore, $a\simeq b$ shall mean that 
for some $C_1,C_2$, we have
$a(x)\le C_1\cdot b(x)\le C_2 \cdot a(x)$ for all $x$. 

{\em Norm estimates. \ }
By $\gamma$ we denote the smallest eigenvalue of $A$ and by $\Gamma$ the greatest
eigenvalue of $A$.

For any vector space norm $\norm{\cdot}$ and any relatively compact
neighborhood $U$ of the origin there exist constants $C_1,C_2>0$ such that we have the norm estimate
\begin{equation}  \label{normab}
\betrag{x}^\Gamma\le C_1 \norm{x}\le C_2\betrag{x}^\gamma \quad \text{for all }  x\in U  \end{equation}
 and \begin{equation}\label{normab2}
\betrag{x}^\gamma\le C_1 \norm{x}\le C_2\betrag{x}^\Gamma \quad \text{for all }  x\in G\setminus U.
\end{equation}

{\em Cancellation. \ } 
The quantity $\betrag{\psi(x)-\psi(y)}$ can be estimated in various ways in terms of the distance between $x$ and $y$,  for example, see \cite[p. 28]{fost}. The following lemma is fine for our purpose.
\begin{lemma} \label{mws}

For any $l\in\N$ there exists a constant $C$ and a 
Schwartz norm $\norm{\cdot}_{(N)}$ such that
for all $\psi\in\schwartz(G)$ and all $x,y\in G$ we have
\[ \betrag{\psi(x)-\psi(y)}\le C \norm{\psi}_{(N)}\cdot \betrag{x^{-1}y}^\gamma\cdot
\left(\frac{1}{(1+\norm{x}_2)^l}+ \frac{1}{(1+\norm{y}_2)^l}\right).
\]
\end{lemma}
 
\begin{proof} 
We find $N$ and $C$ such that the Lemma is true on the additional assumption that $\betrag{x^{-1}y} \ge 1$, because  $\psi\in\schwartz$. We will possibly increase the values of $N$ and $C$ later.
Now let $\betrag{x^{-1}y} \le 1$.
Since the map
\[ h: G\times G \to G \kq (x,z)\mapsto x-xz \]
is smooth, any derivative of $h$ is bounded on any compact set. Furthermore, we have
$h(x,0)=0$. Now compactness and the norm estimate \eqref{normab} yield for any $x$ with $\betrag{x}\le 1$
\begin{equation} \label{a1} \norm{x-y}_2=\norm{h(x,x^\me y)}_2\le C_1 \norm{x^\me y}_2
\le C_2\betrag{x\inv y}^\gamma.
\end{equation}
If $t=\betrag{x}\ge 1$, using \eqref{a1} with $x-y$ replaced by $D_{t\inv}x-D_{t\inv}y$, we obtain 
\begin{equation} \begin{split} \label{a2}
\norm{x-y}_2&=\norm{D_t(D_{t\inv}x-D_{t\inv}y)}_2
\le \norm{D_t}_{\mathrm{op}} \norm{D_{t\inv}x-D_{t\inv}y}_2 \\
&\le Ct^\Gamma \betrag{(D_{t\inv}x)\inv (D_{t\inv}y)}^\gamma
= Ct^{\Gamma-\gamma} \betrag{x\inv y}^\gamma
= C\betrag{x}^{\Gamma-\gamma} \betrag{x\inv y}^\gamma .
\end{split} \end{equation}
Interchanging the roles of $x$ and $y$ in \eqref{a2} and combining with \eqref{a1}, we get

\[ \norm{x-y}_2\le C (1+\min\{ \betrag{x},
 \betrag{y}\}^{\Gamma-\gamma})
 \betrag{x\inv y}^\gamma.\]
With the help of \eqref{normab2} we find some $p\in\N$ such that
$\betrag{z}^{\Gamma-\gamma}\le C \norm{z}_2^p$ for all $z\in G\setminus B_1$.
Setting  $R\dd\min\{ \norm{x}_2,\norm{y}_2 \}$, it follows that
\begin{equation}\label{ab1}
 \norm{x-y}_2\le C (1+R)^p
 \betrag{x\inv y}^\gamma.
\end{equation}

We choose  an integration path from $x$ to $y$ in $M\dd \menge{z\in G}{\norm{z}_2\ge R}$ such that
the length of this path is bounded by a constant multiple of $\norm{x-y}_2$.
Let us estimate $\betrag{\psi(x)-\psi(y)}$ by integrating along that path.  
Since $\psi\in\schwartz$, we find a number $N$ such that 
\begin{equation}\label{ab2}
 \norm{(\nabla\psi)(1+\norm{\cdot}_2)^{l+p}}_\infty\le \norm{\psi}_{(N)}.
\end{equation}
Putting \eqref{ab1} and \eqref{ab2} together, we obtain
\[ \betrag{\psi(x)-\psi(y)}\le
C_1 \norm{(\nabla\psi)1_M}_\infty \norm{x-y}_2
\le C_2\frac{\norm{\psi}_{(N)}}{(1+R)^{l+p}}
(1+R)^p\betrag{x\inv y}^\gamma .
 \]
Finally, because of
\[ \frac{1}{(1+R)^l}\le\frac{1}{(1+\norm{x}_2)^l}+\frac{1}{(1+\norm{y}_2)^l},
\]
we get the desired result.
\end{proof}

\begin{lemma} \label{loesch}
There is a constant $C$ and a
Schwartz norm $\norm{\cdot}_{(N)}$ such that for all
$\phi\in\schwartz(G)$ with $\int \phi = 0$, all $\psi\in\schwartz$, and all $0<s<1$ we have
\[ \norm{\psi\falt\phi_s}_1 \le C \norm{\psi}_{(N)}\norm{\phi}_{(N)} s^\gamma \quad\text{and}\quad
\norm{\psi_s\falt\phi}_1 \le C \norm{\psi}_{(N)}\norm{\phi}_{(N)} s^\gamma. \]
\end{lemma}
\begin{proof}
We give a proof of the first inequality. Using Lemma  \ref{mws} with $l=n+1$, we obtain

\[ \begin{split}
\norm{\psi\falt\phi_s}_1
&= \int  \betrag{\int \psi(xy^\me)\phi_s(y)\,dy}\,dx\\
&= \int \betrag{\int [\psi(xy^\me)-\psi(x)]\phi_s(y)\,dy}\,dx\\
&= \int \int \betrag{\psi(xy^\me)-\psi(x)}
\mal\betrag{\phi_s(y)}\,dy \,dx\\
&\le C \norm{\psi}_{(N)} \int \int \betrag{y}^\gamma
( \frac{1}{(1+\norm{x}_2)^{n+1}}+ \frac{1}{(1+\norm{xy\inv}_2)^{n+1}})
\betrag{\phi_s(y)}\,dy\,dx\\
&\le C \norm{\psi}_{(N)} \int  \betrag{y}^\gamma
\int
( \frac{1}{(1+\norm{x}_2)^{n+1}}+ \frac{1}{(1+\norm{xy\inv}_2)^{n+1}})
\,dx
\betrag{\phi_s(y)}\,dy\\
&\le C \norm{\psi}_{(N)}\left(\int  \frac{1}{(1+\norm{x}_2)^{n+1}}\,dx\right)
\int\betrag{y}^\gamma\betrag{\phi_s(y)}\,dy\\
&\le C \norm{\psi}_{(N)}
\int \betrag{D_s y}^\gamma \phi(y) \,dy \\
&\le C \norm{\psi}_{(N)} s^\gamma
\int \betrag{y}^\gamma \phi(y) \,dy \\
&\le C \norm{\psi}_{(N)}\norm{\phi}_{(N)}s^\gamma .
\end{split} \]
\end{proof}

\section{$\LP^2$ results }
The space $\mathrm L^2(G)$ together with the product $ \spd{f}{g}\dd \int_G f(x)\overline{g(x)}\,\mathrm dx$
is a Hilbert space. This allows
to extend the linear operator $T$ defined by \eqref{deftf} to a bounded operator on $\mathrm L^2(G)$, yielding \eqref{lpestimate} for $p=2$.
Observe that the operators
\begin{equation*} \label{defat}
 A_t:\LP^2(G)\to\LP^2(G)\kq A_tf\dd\psi_t\falt(f\krng\ort{t})\kq t>0.
\end{equation*}
are bounded by Young's inequality, which is valid in the context of homogeneous groups, and by \eqref{rotinv}. Namely,  we have
$ \norm{A_t}\le \norm{\psi}_1$ for every $t>0$; and 
$   A_t^\ast f=  (\psi_t^\ast\falt f)\krng \ortn$, where $ \psi_t^\ast(x)=\overline{\psi(x^\me)}$.\\
The  set $\mathcal E\dd\menge{E\subseteq ]0,\infty[ }{E \text{ measurable  and } \int_E \dht <\infty}$ ordered by inclusion is a directed
set. 

\begin{theorem} \label{opdef}
The net $\left(\int_E  A_t \dht\right)_{E\in \mathcal E}$
converges in the weak operator topology to a bounded linear operator $\widetilde T$ on $\LP^2(G)$,
whose restriction to $\schwartz$ equals $T$.
\end{theorem}
\begin{proof} The main task is to show that there is a  function $h$ such that
\begin{equation}\label{cotlar1}   \norm{A_t A_s^{\ast}}^\halb \le h(t,s) \quad\text{ and }\quad \norm{A_t^{\ast} A_s}^\halb \le h(t,s)
 \end{equation}
and 
\begin{equation} \label{cotlarsassumption}
\sup_{s>0} \int_0^\infty h(t,s)\dht  \  < \infty. 
\end{equation}
Then the proof will be finished by using a continuous version of Cotlar's lemma \cite[Appendix B]{foll1}.
Let $s,t>0$ and $f\in \LP^2(G)$. We have the estimates
\[ \begin{split}
\norm{A_t A^\ast_s f}_2
&=\norm{\psi_t\falt[(\psi^\ast_s\falt f)\krng \ort{s}^\me\krng \ort{t}] }_2 \\
&=\norm{(\psi_t\krng\ort{s-t})\falt(\psi^\ast_s\falt f) }_2\\
&=\norm{[(\psi\krng\ort{s-t})_t\falt\psi^\ast_s]\falt f) }_2\\
&\le\norm{(\psi\krng\ort{s-t})_t\falt\psi^\ast_s}_1
\norm{f}_2,
\end{split} \]
\[ \begin{split}
\norm{A_t^\ast A_s f}_2
&=\norm{(\psi^\ast_t \falt[\psi_s\falt (f\krng \ort{s})]) \krng \ort{t}^\me}_2 \\
&\le\norm{\psi_t^\ast\falt\psi_s}_1\norm{f}_2.
\end{split} \]
Setting
\[ h(t,s) \dd \norm{(\psi\krng\ort{s-t})_t\falt\psi^\ast_s}_1^\halb
+ \norm{\psi_t^\ast\falt\psi_s}_1^\halb, \]
 we obtain \eqref{cotlar1}.
It remains to show \eqref{cotlarsassumption}.
This can be done by using the Schwartz norms of $\psi\krng\ort{s-t}$, which
are bounded uniformly in $s$ and $t$, and Lemma \ref{loesch}. Let us consider an arbitrary  family $\{\phi^{s,t}\}_{s,t>0} $   of Schwartz functions such that 
 $\norm{\phi^{s,t}}_{(N)}\le C_N$ for all $s,t>0$. 
Furthermore, assume $\chi\in\schwartz$ and $\int \phi^{s,t}=\int \chi=0$ for all $s,t>0$. 
Then there exists a Schwartz norm  $\norm{\cdot}_{(N)}$ such that for all $s>0$ the estimate
\[ \begin{split}
\int_0^\infty \norm{\phi^{s,t}_t\falt \chi_s}_1^\halb \dht
&= \int_0^\infty \norm{\phi^{s,ts}_{ts}\falt \chi_s}_1^\halb \dht \\
&= \int_0^\infty \norm{(\phi^{s,ts}_t\falt \chi)_s}_1^\halb \dht \\
&=\int_0^\infty \norm{\phi^{s,ts}_t\falt \chi}_1^\halb \dht \\
&=\int_0^1 \norm{\phi^{s,ts}_t\falt \chi}_1^\halb \dht
+\int_0^1 \norm{\phi^{s,t^\me s}_{t^\me}\falt \chi}_1^\halb \dht \\
&=\int_0^1 \norm{\phi^{s,ts}_t\falt \chi}_1^\halb \dht
+\int_0^1 \norm{\phi^{s,t^\me s}\falt \chi_t}_1^\halb \dht \\
&\le C \left(\sup_{s,t}\norm{\phi^{s,t}}^\halb_{(N)}\right)\norm{\chi}^\halb_{(N)}
\int_0^1 t^\frac{\gamma}{2} \dht\\
&\le C \left(\sup_{s,t}\norm{\phi^{s,t}}^\halb_{(N)}\right)\norm{\chi}^\halb_{(N)}
\end{split} \]
holds. Setting $ \phi^{s,t} \dd \psi\krng\ort{s-t} $, $\chi\dd \psi^\ast$;   and    $ \phi^{s,t} \dd \psi^\ast $, $\chi\dd \psi$; yields  \eqref{cotlarsassumption}. It follows that $\widetilde T=\lim_{E\in \mathcal E}\int_E  A_t \dht$ for some bounded operator $\widetilde T$ on $\LP^2(G)$. 

Let $f$ and $g$ be Schwartz functions.
We have $\norm{  A_tf}_\infty \le \norm{\psi_t}_\infty \norm{f}_1\le C t^{-Q}$ and furthermore, for $0<t<1$,
Lemma \ref{loesch} yields $\norm{A_tf}_1 \le C  t^\gamma\norm{\psi}_{(N)} \norm{f}_{(N)}$. This results in
\[ \int_0^1\int_G\betrag{   t^\me A_t(x)\cdot \overline{g(x)}}\,\mathrm dx\,\mathrm dt<\int_0^1 t^\me\norm{  A_tf}_1\cdot \norm{g}_\infty \,\mathrm dt<\infty, \]
\[ \int_1^\infty\int_G\betrag{   t^\me A_t(x)\cdot \overline{g(x)}}\,\mathrm dx\,\mathrm dt<\int_1^\infty t^\me\norm{  A_tf}_\infty\cdot \norm{g}_1 \,\mathrm dt<\infty, \]
that is, $(t,x)\mapsto \betrag{t^\me A_t(x)\cdot \overline{g(x)}}$ is integrable.
Fubini's Theorem yields
\begin{equation}\begin{split}
&\spd{\widetilde Tf}{g} \\
& =\lim_E \int_E \spd{ A_t f}{g} \dht \\
&= \int_0^\infty\int_G   t^\me A_t(x)\cdot \overline{g(x)}\,\mathrm dx\,\mathrm dt \\
&= \int_G \int_0^\infty  t^\me A_t(x)\,\mathrm dt\cdot \overline{g(x)}\,\mathrm dx \\
&=\spd{Tf}{g}, 
\end{split}\end{equation}
showing $\left.\widetilde T\right\vert_\schwartz=T$.
\end{proof}
From now on we denote $\widetilde T$ by $T$ .

\newcommand{\ig}{\frac{\delta}{R}}
\newcommand{\ib}{^c\!\buu_{\kon\delta}(y)}

\section{A space of homogeneous type}
We now define {\em quasi-balls} $\buu\subseteq G$.
For any $r>0$ and $y\in G$ we set 
\[ \buu_r(y)\dd \bigcup_{s\in[-r,r]}B_r\mal(\ort{s}y)
.\] 
Then by \eqref{rotinvariance} we have  $x\in\buu_r(y)\gdw y\in\buu_r(x)\gdw \exists s: \betrag{s}\le r\wedge\betrag{x\cdot \ort{s}y^{-1}}<r$ and
$\ort{s}(\buu_r(y))=\buu_r(\ort{s}y)$. We show that the balls $\buu$ possess the engulfing and doubling
properties as described by Stein in \cite[p. 8]{stein1}.
\begin{theorem} \label{homraum}
There exist constants $C, k>0$ such that for all $x,y\in G$ and $t>0$
\begin{equation} \label{verschluckung}
\buu_t(y)\cap\buu_t(x)\neq\emptyset\folgt \buu_t(y)\subseteq\buu_{k t}(x) ,
\end{equation}
\begin{equation}
\label{doubling} \leb(\buu_{2t}(x))\le C\cdot\leb(\buu_t(x)).
\end{equation}
\end{theorem}

\begin{proof}
Choose $a_1,\ldots,a_l\in G$ such that  $B_2\subseteq \bigcup_{k=1}^l
 a_k B_1$.
These exist, since $B_2$ is relatively compact and $\{ z\mal B_1 \}_{z\in G}$ is an
open covering of $\overline{B_2}$. It follows that $B_{2t} =D_t B_2 \subseteq
D_t(\bigcup_{k=1}^l a_k B_1) \subseteq \bigcup_{k=1}^l  (D_t a_k)(D_t B_1)
=\bigcup_{k=1}^l (D_t a_k) B_t $. Finally we obtain
\[
\begin{split}
\leb(\buu_{2t}(x)) 
&= \leb(
 B_{2t}\mal\menge{\ort{s}x}{s\in[-2t,2t]}
)  \\
&\le \leb\left(\left(\bigcup_{k=1}^l (D_t a_k)\mal B_t\right) \mal \left(
\bigcup_{\sigma=\pm t}\menge{\ort{s+\sigma}x}{\betrag{s}\le t}
\right) \right) \\
&\le \leb\left( \bigcup_{k=1}^l \bigcup_{\sigma=\pm t} (D_t a_k) \mal B_t
\mal
\menge{\ort{s+\sigma}x}{\betrag{s}\le t} \right) \\
&\le \sum_{k=1}^l \sum_{\sigma=\pm t} \leb \left(\  (D_t a_k) \mal B_t
\mal
\menge{\ort{s+\sigma}x}{\betrag{s}\le t}\  \right) \\
&\le \sum_{k=1}^l \sum_{\sigma=\pm t} \leb\left( \ \ort{-\sigma}(B_t \mal
\menge{\ort{s+\sigma}x}{\betrag{s}\le t}) \  \right) \\
&\le \sum_{k=1}^l \sum_{\sigma=\pm t} \leb(\  B_t \mal
\menge{\ort{s}x}{\betrag{s}\le t} \ )\\
&= 2l \leb(\buu_t(x)).
\end{split} \]
This is the doubling property (\ref{doubling}). The engulfing property \eqref{verschluckung} is known to be true when the quasi-balls
 $\buu_t$ are replaced by the simpler quasi-balls $B_t$. So we may choose
a constant $k\ge 3$ such that \eqref{verschluckung} holds with $B_t$ instead of
$\buu$, and proceed to prove \eqref{verschluckung}.

Now assume that $\buu_t(x)\cap\buu_t(y)\neq\emptyset$. Choose $a,b\in [-t,t]$ such that
$B_t\mal(\ort{a}x)\cap B_t\mal(\ort{b}y)\neq\emptyset$.
Property \eqref{verschluckung} with $B$ instead of $\buu$ yields 
$ B_t(\ort{a}x)\subseteq B_{kt}(\ort{b}y)$. 
Let $s\in [-t,t]$. Rotating both sets with $\ort{s-a}$ and keeping in mind that
$\betrag{s-a+b}\le 3t\le kt$, we conclude that
$ B_t\cdot (\ort{s}x)\subseteq B_{kt}(\ort{s-a+b}y)
\subseteq \buu_{\kon t}(y)$,  that is, $\buu_t(x) \subseteq \buu_{k t}(y)$. 
\end{proof}

\begin{corollary}
The Hardy-Littlewood
maximal operator
\[ (Mf)(x) \dd \sup_{r>0}\frac{1}{\leb({\buu_r(x))}}
\int_{\buu_r(x)}\betrag{f(y)}\,dy\] is of weak type (1,1). 
\end{corollary}
For a proof, see \cite{stein1} for example.
The quasi-balls $\buu$ define a  quasi-metric
\begin{equation}\label{quasi-metric}
 d(x,y)\dd \inf \menge{r>0}{x\in\buu_r(y) }.
\end{equation}
This quasi-metric yields a  space of homogeneous type in the sense of Coifman and Weiss \cite{CW}.
Note that for any quasi-metric there are constants  $C_1,C_2,C_3\ge 1$ such that for all $x,y,\bar y$ we have
\begin{equation}\label{distequiv}
C_1\cdot d(y,\bar y)<d(x,y)\folgt d(x,y)\le C_2 d(x,\bar y)\le C_3 d(x,y).
\end{equation}

\section{The integral kernel}\label{sec:kern}
In this section we study singular integral kernels $K(x,y)$ related to $T$. 
Let $\eta:G\to\C$ be a continuous function such that $\norm{\eta}_\infty=C_1<\infty$ and $\sup_x \betrag{\eta(x)\betrag{x}^N}=C_2<\infty$ for some $N>Q$.
\begin{lemma}\label{lemma:kernelpointwise}
 The integral
\begin{equation}\label{kerneleta}
K_\eta(x,y)=\int_0^\infty \eta_t(x\cdot\ortn y^{-1})\dht\kq x\neq y,\ x,y\in G
\end{equation}
converges and defines a continuous function $K:G\times G\setminus\menge{(x,x)}{x\in G}\to\C$. There exists a constant $C$ such that
\begin{equation}\label{kernelpointwise}
\betrag{K_\eta(x,y)}\le C\cdot [d(x,y)]^{-Q}
\end{equation}
for all $x, y\in G$, $x\neq y$. The estimate \eqref{kernelpointwise} remains true if in \eqref{kerneleta}  $\ortn$ is replaced by $\ortp$ or $x\cdot\ortn y^{-1}$ by $(\ortn x)\cdot y^{-1}$. 
\end{lemma}
\begin{proof}
Assume that $x_0,y_0\in G$ and that $d(x_0,y_0)=\epsilon>0$. Then for any $0\le t\le\epsilon$ we have
\begin{equation*} 
\betrag{x_0\cdot\ortn y_0^{-1}}\ge \epsilon
\end{equation*}
and there is a $\delta>0$ such that 
\begin{equation} \label{argestimate}
\betrag{x\cdot\ortn y^{-1}}\ge \frac{\epsilon}{2} \quad \text{for all}\ x\in B_\delta\cdot x_0, \   y\in B_\delta\cdot y_0 \ \text{ and }\ 0\le t\le\epsilon.
\end{equation}
We will construct a function $H:\R^+\to\R^+$ such that
\begin{equation} \label{majorant}
\betrag{\eta_t(x\cdot\ortn y^{-1})}\le H(t)\quad \text{for all}\ x\in B_\delta\cdot x_0, \ \;  y\in B_\delta\cdot y_0 \ \text{ and }\ 0\le t\le \epsilon
\end{equation}
and such that
\begin{equation}\label{majorantestimate}
\int_0^\infty H(t)\dht \le C \epsilon^{-Q},
\end{equation}
where $C$ is a constant depending on $\eta$, but not on $x$ or $y$.
Then the estimate \eqref{kernelpointwise} is obvious. Furthermore, the convergence of the integral and the continuity of $K$ follow by the dominated convergence theorem. To prove \eqref{majorantestimate}, note that
$\eta_t(x\cdot \ortn y^{-1})=t^{-Q}\eta\krng h(t)$, where $h(t)=\dil_{t^{-1}}(x\cdot \ortn y^{-1})$. 
Because of \eqref{argestimate}, 
  $\betrag{h}$ is bounded  from below by $\betrag {h(t)}=t^{-1}\betrag{x\cdot \ortn y^{-1}}\ge \halb t^{-1}\epsilon\ge \halb$ for $0\le t\le\epsilon$, and $\eta$ is bounded from above by
\begin{equation}
\betrag{\eta\krng h(t)}\le \frac{C_2}{\betrag{h(t)}^N}\le C_2 2^N\cdot t^N\cdot \epsilon^{-N}. 
\end{equation}
Then $H(t)=t^{-Q}\begin{cases} C_1& t>\epsilon\\ C_2 2^N t^N\epsilon^{-N} & t\le\epsilon\end{cases}$ yields
 \eqref{majorant} and \eqref{majorantestimate}, which finishes the proof of Lemma \ref{lemma:kernelpointwise}.
\end{proof}

For example, $\eta=\psi$ and $\eta=\betrag{\psi}$ satisfy the assumptions of Lemma \ref{lemma:kernelpointwise}. Note that \eqref{kernelpointwise} is weaker than the condition 
\[\betrag{K(x,y)}\le \frac{C}{\leb(\buu_{d(x,y)}(y))} \]
required for standard Calder\'on-Zygmund kernels, see Sect. \ref{section:standardkernels}. 

The pointwise estimate \eqref{kernelpointwise} does not suggest integrability of $K(\cdot,y)$:
\[\int_G[d(x,y)]^{-Q}\,\mathrm dx\ge \int_G \betrag{xy^\me}^{-Q}\,\mathrm dx=C\int_0^\infty r^{-Q} r^{Q-1}\,\mathrm dr =\infty.\]
Nonetheless, if $K_\eta=\,^\epsilon\!K_\eta+ K^\epsilon_\eta$, where
\begin{equation}\label{kernelepsilon1}
^\epsilon\! K_\eta(x,y)=\int_0^\epsilon \eta_t(x\cdot\ortn y^{-1})\dht\kq x\neq y,\ x,y\in G,
\end{equation}
\begin{equation}\label{kernelepsilon2}
 K^\epsilon_\eta(x,y)=\int_\epsilon^\infty \eta_t(x\cdot\ortn y^{-1})\dht\kq x,y\in G,
\end{equation}
then $^\epsilon\! K_\eta(\,\cdot\,,y_0)$ is a function that is integrable at infinity and $K^\epsilon_\eta$ is a bounded function such that for every $y_0$ the kernel satisfies some estimate $K^\epsilon_\eta(x,y_0)\lesssim\betrag{x}^{-Q}$ as $x$ tends to infinity.
If $\eta\ge0$ and $\eta(0)>0$, then even $K^\epsilon_\eta(x,y_0)\gtrsim\betrag{x}^{-Q}$ .
\begin{lemma} \label{epskern}
There is a constant $C$ such that for all $y$, $\eta$,  and $\epsilon$, we have
\begin{equation*}
\int_{d(x,y)>\epsilon}\betrag{\, ^\epsilon\!K_\eta(x,y)}\,\mathrm dx \le C\cdot \sup_x \betrag{\eta(x)\betrag{x}^N}.
\end{equation*}

\end{lemma}
\begin{proof} Without loss of generality we assume $\eta\ge 0$ and $C_2$ as before.
Assume that $0<\epsilon\le d(x,y)$ and $0<t<\epsilon$. 
Then we have $\betrag{x\cdot\ortn y^\me}\ge \epsilon$. 
Substituting $z$ for $x\cdot\ortn y^\me$, we obtain $\betrag{z}\ge \epsilon$ and, using spherical coordinates \eqref{spherical}, we obtain the estimate
\begin{equation*}\begin{split}
&\int_{d(x,y)>\epsilon}\int_0^\epsilon \eta_t(x\cdot\ortn y^{-1})\dht\,\mathrm dx\\
&\le\int_0^\epsilon \int_{\betrag{z}\ge\epsilon}\eta_t(z)\,\mathrm dz\dht\\
&\le\int_0^\epsilon \int_{\betrag{z}\ge\frac{\epsilon}{t}}\eta(z)\,\mathrm dz\dht\\
&\le \sigma(S_1) \int_0^\epsilon \int_{\frac{\epsilon}{t}}^\infty \frac{C_2}{r^N} r^{Q-1} \,\mathrm dr\dht\\
&\le C\cdot C_2 \int_0^\epsilon  \left(\frac{\epsilon}{t}\right)^{Q-N} \dht\ \\
&\le C\cdot C_2.
\end{split}\end{equation*}
\end{proof}

\begin{lemma} The operator $T$ is expressible as a singular integral as follows.
 For all $f\in\LP^2(G)$ with compact support and all $x\in G\setminus \supp f$, the integrals
\begin{equation*}
\label{opint}
Tf(x) = \int_G K_\psi(x,y)f(y)\,\mathrm dy \quad \text{and}\quad 
T^\ast f(x) = \int_G \overline{K_\psi(y,x)}f(y)\,\mathrm dy
\end{equation*}
converge and equality holds for almost all $x\in G\setminus
\supp f$.
\end{lemma}
\begin{proof} The function $K_{\betrag{\psi}}$ is continuous by Lemma \ref{lemma:kernelpointwise}.  
If $f,g\in\LP^2(G)$ have compact support and $\supp f\cap\supp g =\emptyset$, then we have the estimate
\begin{equation*}\begin{split}
\infty>&\int_G\int_G K_{\betrag{\psi}}(x,y)\betrag{f(y)}\betrag{g(x)} \,\mathrm dy\,\mathrm dx \\
&= \int_G\int_G \int_0^\infty \betrag{\psi_t(x\cdot\ortn y^\me)\cdot f(y)\cdot\overline{g(x)}}\dht\,\mathrm dx\,\mathrm dy.
\end{split}\end{equation*}
Tonelli's and Fubini's Theorems imply that 
\begin{equation*}\begin{split}
 &\int_G \int_G K_\psi(x,y) f(y) \,dy\  \ov{g(x)}\,dx \\ 
&= \int_G\int_G \int_0^\infty \psi_t(x\cdot\ortn y^\me) \cdot f(y)\cdot\ov{g(x)} \dht\,\mathrm dx\,\mathrm dy \\
&= \int_0^\infty\int_G\int_G  \psi_t(x\cdot\ortn y^\me) \cdot f(y)\,\mathrm dy \ \ov{g(x)} \,\mathrm dx \dht \\
&= \int_0^\infty\int_G\int_G  \psi_t(x\cdot y^\me) \cdot f(\ortp y)\,\mathrm dy\ \ov{g(x)} \,\mathrm dx \dht \\
&= \int_0^\infty \spd{\psi_t\ast(f\krng\ortp)}{g} \dht \\
&= \int_0^\infty \spd{A_t f}{g} \dht \\
&= \spd{Tf}{g},
\end{split}\end{equation*}
where $\ortp y$ has been substituted for $y$. 
A similar calculation yields
\begin{equation*}
 \int_G \int_G \overline{K_\psi(y,x)} f(y) \,\mathrm dy\  \ov{g(x)}\,\mathrm dx = \spd{T^\ast f}{g}.
\end{equation*}
Finally, let $\overline B\subseteq G$ be a compact ball with rational radius and rational center such that $\overline B\subseteq G\setminus \supp f$.
Then $ \int_G K_\psi(x,y) f(y) \,dy$ converges for all $x\in\overline B$,
\begin{equation*}
\spd{Tf}{g}= \spd{ \int_G K_\psi(\,\cdot\,,y) f(y) \,\mathrm dy }{g}
\end{equation*}
holds for all $g\in\LP^2(B)$ and hence $Tf(x)=\int_G K_\psi(x,y) f(y) \,\mathrm dy$ for almost all $x\in\overline B$. Since there are only 
countably many balls $\overline B$ and every $x\in G\setminus \supp f$ is contained in such a ball, we have equality for almost all $x\in G\setminus\supp f$.

Similar arguments apply to $T^\ast$.
\end{proof}

We will need the following technical lemma.
\begin{lemma}\label{tecnical}
There is a constant $C_1>0$ such that, for all $x\in G$, $t>0$ and $s\in[-\frac{t}{2},\frac{t}{2}]$,
\[ \betrag{ 
(\dil_{t-s}^\me x)^\me(\dil_{t}^\me x)}\le C_1 \betrag{\dil_{t}^\me x}\cdot\left(\frac{\betrag{s}}{t}\right)^\frac{1}{\Gamma}. \]
\end{lemma}
\begin{proof}
The map
\[ h:\menge{x\in G}{\betrag{x}=1}\times[-\frac{1}{3},2]\to G,\quad (x,p)\mapsto x\inv D_{1+p}x \]
is the restriction of a smooth map to a compact set. Since furthermore $h(x,0)=0$,
we find a constant $C$ such that for  all  $p\in[-\frac{1}{3},2]$ we have
\[ \betrag{x}=1 \folgt \norm{x\inv D_{1+p}x}_2\le C\betrag{p}. \]
The norm estimate \eqref{normab} yields
\[ \betrag{x}=1 \folgt \betrag{x\inv D_{1+p}x}\le C\betrag{p}^\frac{1}{\Gamma} . \]
Setting $p\dd\frac{t}{t-s}-1=\frac{s}{t-s}$, we obtain the estimate
\[ \betrag{x}=1\  \folgt\  \betrag{x\inv D_\frac{t}{t-s}x}\le C\betrag{\frac{s}{t-s}}^\frac{1}{\Gamma}
\le C_1 \betrag{\frac{s}{t}}^\frac{1}{\Gamma},\]
and for arbitrary $x$ with $\betrag{x}=r$, we have
\[  C_1 \betrag{\frac{s}{t}}^\frac{1}{\Gamma}
\ge  \betrag{(\dil_r^\me x^\me) D_\frac{t}{t-s}\dil_r^\me x}
= r^\me t \betrag{( \dil_t^\me x^\me)( D_{t-s}^\me x)}.
\]
Finally, 
multiplying with $r\cdot t^\me$ yields
\[  C_1 \betrag{\frac{s}{t}}^\frac{1}{\Gamma}t^\me\betrag{x}
\ge \betrag{( \dil_t^\me x^\me)( D_{t-s}^\me x)}=\betrag{ 
(\dil_{t-s}^\me x)^\me(\dil_{t}^\me x)}.
\]

\end{proof}

\begin{theorem} \label{intab}
There are constants $k$ and 
  $C$ and
a Schwartz norm $\norm{\cdot}_{(N)}$  such that for all $\psi\in\schwartz(G)$,
$\delta>0$, $y\in G$, $\bar y\in\buu_\delta(y)$ it holds that
\begin{equation}\label{zugl}
\int_{G\setminus\buu_{\kon\delta}(y)} \betrag{K_\psi(x,y)-K_\psi(x,\bar y)} \,
\mathrm dx \le \ C\norm{\psi}_{(N)}.
\end{equation}
 If in addition 
$\psi(O_t x)=\psi(x)$ for all $t$ and $x$, then
\begin{equation}\label{zugl*}
\int_{G\setminus\buu_{\kon\delta}(y)}
\betrag{\ov{K_\psi(y,x)}-\ov{K_\psi(\bar y,x)}} \, \mathrm dx  \le \ C\norm{\psi}_{(N)}.
\end{equation} 
\end{theorem}
\begin{proof}
Given the assumption about $\psi$, we have
\[ 
 \overline{K(y,x)} = \int_0^\infty 
\psi^\ast_t (x\cdot\ortp y^\me ) \dht.
\] 
Therefore,  \eqref{zugl*} can be proven like \eqref{zugl} with minor changes. We prove \eqref{zugl}:
Because of \eqref{distequiv}, there is some $k\ge 3$ such that $d(x,y)\ge k\delta\wedge d(y,\bary)<\delta \folgt d(x,\bary)\ge 2\delta$ for every $\delta>0$. 
Let  $\delta>0$ and $y,\bary\in G$ such that $d(y,\bary)<\delta$. Let $s$ be a number such that $\betrag{\ortps y\bary^\me}<\delta$ and $\betrag{s}\le\delta$.
Lemma \ref{epskern} yields the estimate 
\begin{equation}\label{singterm}
\int_{d(x,y)>k\delta}\betrag{\,^{2\delta-s}\!K_\psi(x,y)-\,^{2\delta}\!K_\psi(x,\bary)}\le C \norm{\psi}_{(N_0)},
 \end{equation}
for some Schwartz norm $\norm{\,\cdot\,}_{(N_0)}$. Hence, it is sufficient to show that 
\begin{equation}\label{toshow}
\int_G\betrag{K^{2\delta-s}_\psi(x,y)-K^{2\delta}_\psi(x,\bary)}\,\mathrm dx\le C \norm{\psi}_{(N')}.
 \end{equation}
Observe that by substituting $t-s$ for $t$, we have
\begin{equation}
K^{2\delta-s}_\psi(x,y)=\int_{2\delta-s}^\infty \psi_t(x\cdot\ortn y^{-1})\dht=\int_{2\delta}^\infty \frac{\psi_{t-s}(x\cdot\ort{-t+s}y^{-1})}{t-s} \,\mathrm dt.
 \end{equation}
This transformation yields additional cancellation because afterward, in \eqref{toshow}, at the same value of $t$, the function $\psi$ is evaluated in two points of small homogeneous distance.
To improve the estimates, we decompose the kernel integral in \eqref{toshow} further.
Setting 
\[ a(t,x)= \frac{\psi_{t-s}(x\cdot\ort{-t+s}y^{-1})}{t-s}
\ \text{ and } \  d(t,x)=\frac{\psi_{t}(x\cdot\ort{-t}\bary^{-1})}{t},
\]
\[ b(t,x)= \frac{\psi_{t-s}(x\cdot\ort{-t}\bary^{-1})}{t-s}
\ \text{ and } \  c(t,x)=\frac{\psi\krng\dil_{t-s}^\me (x\cdot\ort{-t}\bary^{-1})}{t^{Q+1}}, 
\]
 we have 
\begin{equation}\begin{split} \label{split}
&\betrag{K^{2\delta-s}_\psi(x,y)-K^{2\delta}_\psi(x,\bary)}\\
&\le \int_{2\delta}^\infty \betrag{ a(t,x)-d(t,x)}\,\mathrm dt \\
& = \int_{2\delta}^\infty \betrag{ a(t,x)-b(t,x)+b(t,x)-c(t,x)+c(t,x)-d(t,x)}\,\mathrm dt \\
&\le \int_{2\delta}^\infty \betrag{ a(t,x)-b(t,x)}\,\mathrm dt 
+\int_{2\delta}^\infty \betrag{ b(t,x)-c(t,x)}\,\mathrm dt
+\int_{2\delta}^\infty \betrag{ c(t,x)-d(t,x)}\,\mathrm dt.
\end{split}\end{equation} 
Note that in these integrals $t-s\simeq t$.
The proof will be completed by showing that the estimate \eqref{toshow} holds with the integrand replaced by each of the three $t$-integrals 
at the end of \eqref{split}. 
The first of these integrals is
\begin{equation}\begin{split}
&\int_G \int_{2\delta}^\infty \betrag{ a(t,x)-b(t,x)}\,\mathrm dt\,\mathrm dx\\
&\le C\int_{2\delta}^\infty \int_G t^{-Q} \betrag{\psi\krng\dil_{t-s}^\me(x\cdot\ort{-t+s}y^{-1}) -\psi\krng\dil_{t-s}^\me(x\cdot\ort{-t}\bary^{-1})} \,\mathrm dx\dht,
\end{split}\end{equation} 
 which, by Lemma \ref{mws}, is bounded by
\begin{equation}\label{dfd}\begin{split}
& C\int_{2\delta}^\infty \int_G  t^{-Q} \ C\norm{\psi}_{(N_1)} \betrag{\dil_{t-s}^\me [(\ortps y)\bary^\me]}^\gamma  H(x,t) \ \mathrm dx\dht,
\end{split}\end{equation} 
where
\begin{equation*}
H(x,t
)=\left(\frac{1}{(1+\norm{\dil_{t-s}^\me[x\cdot\ort{-t+s}y^{-1}]}_2)^l}+ \frac{1}{(1+\norm{\dil_{t-s}^\me[ x\cdot\ort{-t}\bary^{-1}]}_2)^l}\right)
\end{equation*} 
 and $l>n$. Note that there is a $C$ such that for every $t$
\begin{equation*}
\int_G H(x,t) \,\mathrm dx\le C t^Q .
\end{equation*} 
Hence \eqref{dfd} continues as
\begin{equation}\tag{\ref{toshow}ab}\begin{split}
&\int_G \int_{2\delta}^\infty \betrag{ a(t)-b(t)}\,\mathrm dt\,\mathrm dx\\
&\le C\norm{\psi}_{(N_1)}\int_{2\delta}^\infty  t^{-Q}\, t^Q\,  \betrag{\dil_{t-s}^\me [(\ortps y)\bary^\me]}^\gamma \dht  \\
& \le C\norm{\psi}_{(N_1)}\int_{2\delta}^\infty  \ t^{-\gamma} \delta^\gamma \dht  \\
& \le C\norm{\psi}_{(N_1)}.
\end{split}\end{equation} 
The second integral can be written as
\begin{equation*}\begin{split}
&\int_G \int_{2\delta}^\infty \betrag{ b(t,x)-c(t,x)}\,\mathrm dt\,\mathrm dx\\
&=\int_G \int_{2\delta}^\infty \betrag{  \frac{\psi_{t-s}(x\cdot\ort{-t}\bary^{-1})}{t-s}
-\frac{\psi\krng\dil_{t-s}^\me (x\cdot\ort{-t}\bary^{-1})}{t^{Q+1}} } \\
&= \int_{2\delta}^\infty \int_G\betrag{(t-s)^{-Q-1}-t^{-Q-1}} \cdot\betrag{\psi\krng\dil_{t-s}^{-1}(x\cdot\ort{-t}\bary^{-1})} \,\mathrm dx\,\mathrm dt\\
&=\int_{2\delta}^\infty \betrag{(t-s)^{-Q-1}-t^{-Q-1}} \int_G  \cdot\betrag{\psi\krng\dil_{t-s}^{-1}(x)} \,\mathrm dx \,\mathrm dt.
\end{split}\end{equation*} 
This yields the estimate
\begin{equation}\tag{\ref{toshow}bc}\begin{split}
&\int_G \int_{2\delta}^\infty \betrag{ b(t,x)-c(t,x)}\,\mathrm dt\,\mathrm dx\\
&\le C\int_{2\delta}^\infty t^{Q-2} \betrag{s} \int_G  \cdot\betrag{\psi\krng\dil_{t-s}^{-1}(x)} \,\mathrm dx \,\mathrm dt\\
&\le C\int_{2\delta}^\infty t^{Q-1} \delta \int_G  \cdot\betrag{\psi\krng\dil_{t-s}^{-1}(x)} \,\mathrm dx \dht\\
&\le C\int_{2\delta}^\infty t^{Q-1} \delta \  \cdot (t-s)^Q\norm{\psi}_1 \dht\\
&\le C \norm{\psi}_1 \delta\int_{2\delta}^\infty t^{-1}   \dht\\
&\le C \norm{\psi}_1\\
&\le C \norm{\psi}_{(N_2)}.
\end{split}\end{equation} 
This was the second part. The third part is
\begin{equation}\begin{split}\label{ugk1}
&\int_G \int_{2\delta}^\infty \betrag{ c(t,x)-d(t,x)}\,\mathrm dt\,\mathrm dx\\
&=\int_G \int_{2\delta}^\infty \betrag{ 
\frac{\psi\krng\dil_{t-s}^\me (x\cdot\ort{-t}\bary^{-1})}{t^{Q+1}} - \frac{\psi_{t}(x\cdot\ort{-t}\bary^{-1})}{t} }\,\mathrm dt\,\mathrm dx\\
&= \int_{2\delta}^\infty \int_G t^{-Q}\betrag{ 
\psi\krng\dil_{t-s}^\me (x\cdot\ort{-t}\bary^{-1}) - \psi\krng\dil_{t}^\me(x\cdot\ort{-t}\bary^{-1})} \mathrm dx\dht \\
&= \int_{2\delta}^\infty \int_G t^{-Q}\betrag{ 
\psi(\dil_{t-s}^\me x) - \psi(\dil_{t}^\me x)} \mathrm dx\dht \\
&\le C\norm{\psi}_{(N)} \int_{2\delta}^\infty \int_G t^{-Q}\betrag{ 
(\dil_{t-s}^\me x)^\me(\dil_{t}^\me x)}^\gamma  H(x,t)\mathrm dx\dht. \\
\end{split}\end{equation} 
The last step relies on Lemma \ref{mws}, where  $l$ has been chosen such that $\int_G\betrag{\dil_{t}^\me x}^\gamma \cdot H(x,t)\,\mathrm dx<\infty$,  and   $H$ is the function
\begin{equation}\label{Hab}
H(x,t)= \frac{1}{(1+\norm{\dil_{t-s}^\me x}_2)^l}+ \frac{1}{(1+\norm{\dil_{t}^\me x)}_2)^l}.
\end{equation} 
Since $\betrag{s}\le \delta$ and $t\ge 2\delta$, we have $\norm{\dil_\frac{t}{t-s}}\le C$ and $\norm{\dil_\frac{t-s}{t}}\le C$, that is, 
$\norm{\dil_{t}^\me x}_2\simeq \norm{\dil_{t-s}^\me x}_2$, and the two summands in \eqref{Hab} are bounded by a constant multiple of the second one. Furthermore, 
\[ \int_G\betrag{\dil_{t}^\me x}^\gamma \cdot H(x,t)\,\mathrm dx \le Ct^Q. \]
Using  Lemma \ref{tecnical},
we continue estimate \eqref{ugk1} with 
\begin{equation} \tag{\ref{toshow}{cd}} 
\begin{split}
&\int_G \int_{2\delta}^\infty \betrag{ c(t,x)-d(t,x)}\,\mathrm dt\,\mathrm dx\\
&\le C\norm{\psi}_{(N_3)}  \int_{2\delta}^\infty \int_G t^{-Q}  \left( C_1 \betrag{\dil_{t}^\me x}\cdot\left(\frac{\betrag{s}}{t}\right)^\frac{1}{\Gamma}\right) ^\gamma  H(x,t)\mathrm dx\dht \\
&\le C\norm{\psi}_{(N_3)}  \int_{2\delta}^\infty  t^{-Q-\frac{\gamma}{\Gamma }} \delta^\frac{\gamma}{\Gamma }\int_G  \betrag{\dil_{t}^\me x}^\gamma  H(x,t)\mathrm dx \dht\\
&\le C\norm{\psi}_{(N_3)}  \int_{2\delta}^\infty  t^{-Q-\frac{\gamma}{\Gamma }} \delta^\frac{\gamma}{\Gamma } t^Q\dht\\
&\le C\norm{\psi}_{(N_3)}  \delta^\frac{\gamma}{\Gamma } \int_{2\delta}^\infty  t^{-\frac{\gamma}{\Gamma }}  \dht\\
&\le C\norm{\psi}_{(N_3)} .
\end{split}\end{equation} 
Adding (\ref{toshow}ab),(\ref{toshow}bc),(\ref{toshow}cd)  finishes the proof of \eqref{toshow} with $N'=\max\{N_1,N_2,N_3\}$.
\end{proof}

\section{Proof of Theorem 1}
\begin{proof}
The weak type result \eqref{weaktype} and the $\mathrm L^p$-result \eqref{lpestimate} for $1<p<2$ follow with the help of  \cite[Theorem 3, p.\,19]{stein1}, which relies on a Calder\' on-Zygmund decomposition of $f$, here with respect to the
 quasi-balls $\buu$. The relevant premises have been verified in Theorems \ref{opdef}, \,\ref{homraum},  and  \ref{intab}.

If $2<p<\infty$,\  let  $q\dd\frac{p}{p-1}$ and  $M\dd\menge{g\in\schwartz(G)}{\norm{g}_q<1}$. Note that $q<2$.
With the same arguments as in the proof of boundedness for $T$, but replacing $\overline{K(y,x)}$ for $K(x,y)$, it can be shown that
$\norm{T^\ast g}_q\le C\norm{g}_q$.\,
It follows that
\[ \norm{Tf}_p = \sup_{g\in M} \spd{Tf}{g}
= \sup_{g\in M} \spd{f}{T^\ast g}
\le \sup_{g\in M} \norm{f}_p\norm{T^\ast g}_q
\le C_p \norm{f}_p \]
for any $f\in\LP^2\cap\LP^p(G)$.
\end{proof}

\section{Example}\label{sec:example}
Let $G_1$ be the real vector space $\C\times\R^{n-2}$. We denote its elements by 
\begin{equation*}\label{G1start}
x =(u, x') \kq u\in\C \kq x'=(x_3,\ldots, x_n) \in\R^{n-2}.
\end{equation*}
Let $a,a_3,\ldots a_n>0$. We introduce a family of dilations with $ Q= 2a+\sum_{j=3}^n a_j$ by 
\begin{equation*}
D_rx\dd (r^a u,r^{a_3}x_3,\ldots, r^{a_n} x_n)
\end{equation*}
and a family of rotations by $\ort{t}(u,x')\dd (e^{it}u, x')$.
There are several ways to define homogeneous norms on $G_1$, each of which has its own advantages. 
It is known that on any homogeneous group any two homogeneous norms $\betrag{\,\cdot\,}$ and $\betrag{\,\cdot\,}'$ are
equivalent in the sense that $\betrag{\,\cdot\,} \simeq \betrag{\,\cdot\,}'$. The volumes of the corresponding balls and quasi-balls are also equivalent in the sense that $\mu(\buu_r\cdot z)\simeq \mu(B_r'\cdot z)$ and $\mu(\buu_r(z))\simeq \mu(\buu_r'(z))$ if the balls $B, \buu$ and $B',\buu'$ are defined in terms of the homogeneous norms $\betrag{\,\cdot\,}$ and  $\betrag{\,\cdot\,}'$ respectively.  Let us consider the homogeneous norms
\begin{equation}\label{hnorm1}
\betrag{x}=\max\{ \betrag{u}_{\C}^\frac{1}{a}, \betrag{x_3}_\R^\frac{1}{a_3},\ldots,\betrag{x_n}_\R^\frac{1}{a_n}  \},
\end{equation}
\begin{equation*}
\betrag{x}'= \inf\menge{r>0}{\norm{\dil_{r^{-1}}(x) }_2<1},
\end{equation*}
\begin{equation*}
\betrag{x}''=\left( \betrag{u}_{\C}^\frac{2}{a}+\sum_{j=3}^n \betrag{x_j}^\frac{2}{a_j} \right)^\frac{1}{2}.
\end{equation*}
The second one has the advantage of being smooth away from the origin, while the first one allows effortless calculations.
Throughout this section, we will use $\betrag{\,\cdot\,}$, and theorems will be valid for any homogeneous norm. 

If $y=(v,y')\in G_1$, and if we use the homogeneous norm given by \eqref{hnorm1},  the quasi-metric \eqref{quasi-metric} has the form 
\begin{equation*}\begin{split} 
d(x,y)&= \inf\menge{r>0}{ \exists s:  \betrag{s}\le r\wedge\betrag{x-\ort{s}y}<r}, \\
&= \min_{s\in\R}\max \left\{ \betrag{s},\,  \betrag{u-e^{is}v}_{\C}^\frac{1}{a},\, \betrag{x_3-y_3}^\frac{1}{a_3},\, \betrag{x_n-y_n}^\frac{1}{a_n} \right\}  .
\end{split}\end{equation*}
 With $y$ and $r$ fixed, we have
\begin{equation}\label{fgdar}\begin{split} 
&x\in\buu_r(y) \\
&\gdw d(x,y)<r\\
&\gdw  \max\left\{ \min_{\betrag{s}\le r} \betrag{u-e^{is}v}_{\C}^\frac{1}{a},\,\betrag{x_3-y_3}_\R^\frac{1}{a_3},\, \betrag{x_n-y_n}_\R^\frac{1}{a_n} \right\}<r.
\end{split}
\end{equation}
\begin{lemma} \label{volbal}
For all $y=(v,y')$ in $G_1$ and all $r>0$, the volume $\leb(\buu_r(y))$ is bounded from above and from below by a constant multiple of  
\begin{equation*}
 r^Q+  r^{Q-a}\betrag{v}\min\{1,r\}.
\end{equation*}
\end{lemma}
\begin{proof}
As a first step, we consider the case $n=2$.
\begin{figure*}        
 \includegraphics[width=0.32\textwidth]{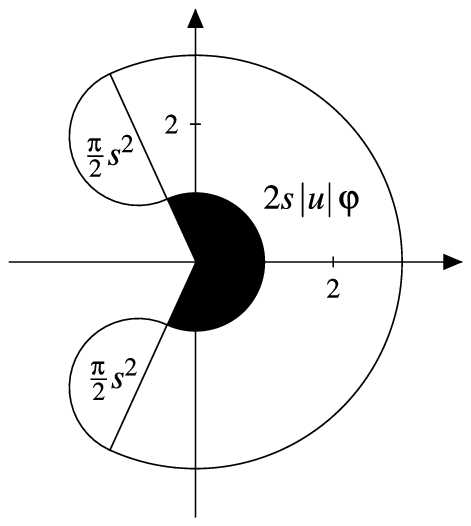} 
\includegraphics[width=0.32\textwidth]{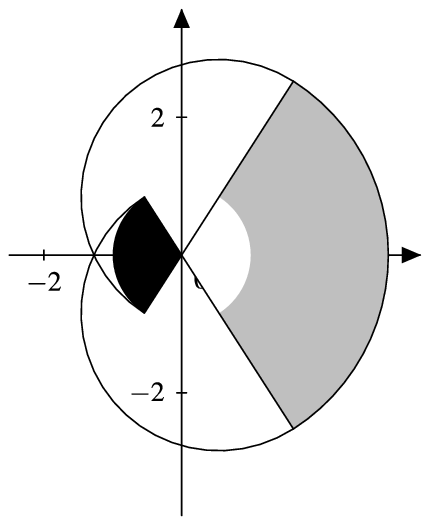}
\includegraphics[width=0.32\textwidth]{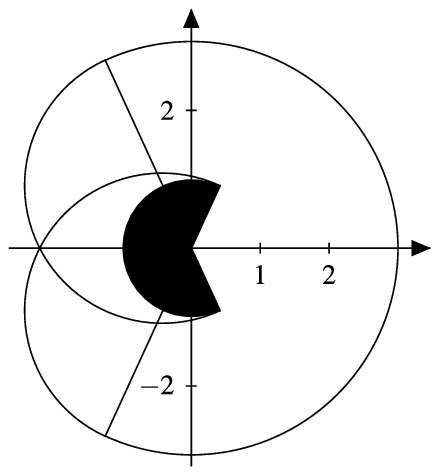}
\caption{The sets $B'(1,2,2)$,\ \  $B'(2,1,1)$ and \ $B'(2,2,1)$. The size of the black areas is $\pi(\betrag{u}-s)^2 \frac{\min\{2\pi,\phi\}}{2\pi}$, see equation \eqref{flab}.}
\label{fig:2}
\end{figure*}
For $v\in\C$, $\phi\in\R$, and $r>0$ let 
\[  B'(s,\phi,v) \dd \bigcup_{\alpha\in[-\phi,\phi]} \menge{z}{\norm{z}_2<s}+e^{i\alpha}v. \]
By some calculations corresponding  to Fig. \ref{fig:2}, we see that 
\begin{equation}\label{flab}
\begin{split}
\leb(B'(s,\phi,v))&\le \pi s^2+ \pi(\betrag{v}+s)^2\frac{\min\{2\pi,\phi\}}{2\pi}-\pi(\betrag{v}-s)^2 \frac{\min\{2\pi,\phi\}}{2\pi} \\
&= \pi s^2+ 2 s\cdot\betrag{v}\min\{2\pi,\phi\}\\
&\le 2 \leb(B'(s,\phi,v)).\end{split} \end{equation}
Equality holds in the first two lines of \eqref{flab} if $s<\betrag{v}$ and $\phi<\frac{\pi}{2}$.
Setting $s\dd r^a$ and $\phi=r$, \eqref{flab} yields 
\begin{equation} \label{bsformel}
\leb(\buu_r(y))=\leb(B'(r^a,r,v))\simeq\pi r^{2a}+ 2 r^a\cdot\betrag{v}\min\{2\pi,r\},
\end{equation} so we are done with the case $n=2$.
Now let $n\ge 3$. Since
\[ 
\buu_r(y)=B'(r^a, r, y)\times \prod_{j=3}^n (y_j-r^{a_j},y_j+r^{a_j})
\]
we have
\[ \begin{split}
\mu(\buu_r(y))&=\mu(B'(r^a, r, y))\cdot (2r)^{Q-2a} \\
& \simeq \left( r^{2a}+  r^a\cdot\betrag{v}\min\{1,r\} \right) r^{Q-2a},
\end{split}\]
which completes the proof.
\end{proof}

The following Lemma roughly says that in $G_1$, the volume of the balls $\buu$ grows at least as in spaces of dimension $Q-a$.  
\begin{lemma}\label{volwachstum}
There is a constant $C$ such that for any $y\in G_1$, $r\ge 0$ and $j\in\N$ we have
\begin{equation*}
\leb(\buu_{2^jr}(y))\ge C 2^{j(Q-a)}\leb(\buu_r(y)).
\end{equation*}
\end{lemma}
\begin{proof}
\end{proof} Using Lemma \ref{volbal} we obtain the estimate
\begin{equation}
\begin{split}
\leb(\buu_{2^jr}(y))&\simeq   (2^jr)^Q+  (2^jr)^{Q-a}\betrag{v}\min\{1,2^jr\}   \\
&= 2^{j(Q-a)} \left(  2^{ja} r^Q+  r^{Q-a}\betrag{v}\min\{1,2^jr\} \right)  \\
&\ge 2^{j(Q-a)} \left(  r^Q+  r^{Q-a}\betrag{v}\min\{1,r\} \right)  \\
&\simeq C 2^{j(Q-a)}\leb(\buu_r(y)).
\end{split}
\end{equation}
\begin{lemma}\label{dichte} There is a constant $C_1$ such that 
for all $R\ge1,\, r>0$ and $x,y\in G_1$ we have the estimate
\begin{equation}\label{diab} 
\int_r^{2r} 1_{\menge{t}{\betrag{x-\ortn y}<Rt}}\dht \le C_1 
\frac{R^ar^Q}{\leb(\buu_r(y))}. 
 \end{equation}
Furthermore, for any $p>a-Q$ there is a constant $C_2$ such that
\begin{equation}\label{intabpartb}
 \int_r^\infty t^{-Q-p}\cdot 1_{\menge{t}{\betrag{x-\ortn y}<Rt}}\dht 
\le C_2 \frac{R^ar^{-p}}{\leb(\buu_r(y))} \end{equation}
for all $r>0$ and all $x,y\in G_1$. These statements are also true when $\ortp$ is substituted for $\ortn$ or when $x$ and $y$ are exchanged on one side.
\end{lemma}
\begin{proof}
For some constants $C_1,C_2$ we have $\leb(\buu_r(y))\le  C_1r^Q+  C_2r^{Q-a}\betrag{v}\min\{1,r\}$ by Lemma \ref{volbal}.
If $\leb(\buu_r(y))< 2C_1r^Q$, then we are done with \eqref{diab} because the left side of \eqref{diab} is bounded by a constant. Otherwise we have $\leb(\buu_r(y))\ge 2C_1r^Q$ and
\[ \leb(\buu_r(y)) \le 2C_2 r^{Q-a}\betrag{v}\min\{1,r\}, \]
where $C_2$ does not depend on $y$ or $r$. Note that $v\neq 0$.
Assume for the moment that
\begin{equation}\label{assump1}
\int_r^{2r} 1_{\menge{t}{\betrag{x-\ortn y}<Rr}}\diff t \lesssim \frac{R^a r^{a+1}}{\betrag{v} \min\{ r,1\}}. 
\end{equation}
Then, with $2R$ substituted for $R$, we get
\begin{equation*}\begin{split}
\int_r^{2r} 1_{\menge{t}{\betrag{x-\ortn y}<Rt}}\dht &\le r^{-1}\int_r^{2r} 1_{\menge{t}{\betrag{x-\ortn y}<R2r}}\diff t \\
& \lesssim\frac{(2R)^ar^{Q}}{r^{Q-a}\betrag{v} \min\{ r,1\}}\\
&\lesssim \frac{R^a r^Q}{\leb(\buu_r(y))}, 
\end{split}\end{equation*}
and the proof of \eqref{diab} would be finished. So let us prove assumption \eqref{assump1}.
 Note that by  \eqref{hnorm1}
we have 
\[\betrag{x-\ortn y}<Rr\quad \folgt\quad  \betrag{u-e^{-it}v}^{\frac{1}{a}}_\C< Rr. \]
For any $0<r<r'\le r+2\pi$, it follows with the help of Fig. \ref{fig:3} that
\begin{equation}\begin{split} \label{ea}
\int_r^{r'} 1_{\menge{t}{\betrag{x-\ortn y}<Rr}}  \diff t
&\le \leb \menge{t\in[0,2\pi]}{\  \betrag{u-e^{-it}v} <(Rr)^{a}}\\
&= \leb \menge{t\in[0,2\pi]}{\  \betrag{\frac{u}{\betrag{v}}-e^{-it}\frac{v}{\betrag{v}}}<\frac{ (Rr)^a}{\betrag{v}}}\\
& \le2\pi\frac{(Rr)^a}{\betrag{v}}\\
& \lesssim\frac{R^a r^a}{\betrag{v}}.\\
\end{split}\end{equation}
\begin{figure*}        
 \includegraphics[width=0.6\textwidth]{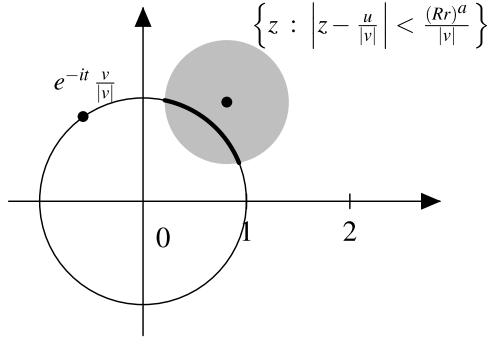} 
\caption{The length of the thick line is smaller than the circumference of the shaded area, see \eqref{ea}.}
\label{fig:3}
\end{figure*}
If $r<1$, then \eqref{assump1} follows from \eqref{ea} with $r'=2r$ because their right hand sides are equal. Otherwise, when $r\ge 1$, we decompose the integral into the sum
\[\begin{split}  
\int_r^{2r}  1_{\menge{t}{\betrag{x-\ortn y}<Rr}}\diff t 
&\le \sum_{k=0}^{\lfloor r\rfloor}  \int_{r+k}^{r+k+1} 1_{\menge{t}{\betrag{x-\ortn y}<Rr}}\diff t \\
&\lesssim \sum_{k=0}^{\lfloor r\rfloor} \frac{R^a\ (r+k)^a}{\betrag{v}}\\
&\lesssim  \frac{R^a r^{a+1}}{\betrag{v}},
 \end{split} \]
 which proves  \eqref{assump1} and finishes the proof of \eqref{diab}. 
A  dyadic decomposition of the interval $]0,\infty[$,  equation \eqref{diab} and Lemma \ref{volwachstum} now yield 
\[\begin{split} 
\int_r^\infty t^{-Q-p}\cdot 1_{\menge{t}{\betrag{x-\ortn y}<Rt}}\dht 
&=\sum_{j=0}^\infty \int_{2^jr}^{2^{j+1}r} t^{-Q-p}\cdot 1_{\menge{t}{\betrag{x-\ortn y}<Rt}}\dht \\
&\le \sum_{j=0}^\infty   (2^{j} r)^{-Q-p} \cdot C_1\cdot \frac{R^a (2^j r)^Q}{\leb(\buu_{2^j r}(y))} \\
&= C_1 R^a r^{-p}\sum_{j=0}^\infty \frac{ 2^{-jp}}{{\leb(\buu_{2^j r}(y))}} \frac{\leb(\buu_r(y))}{\leb(\buu_r(y))}  \\
&\lesssim \frac{R^a r^{-p}}{\leb(\buu_r(y))} \sum_{j=0}^\infty 2^{-jp}\cdot 2^{j(a-Q)}  \\
&\lesssim \frac{R^a r^{-p}}{\leb(\buu_r(y))},
\end{split}\]
which proves \eqref{intabpartb}. 
\end{proof}
\section{Calder\'on-Zygmund kernels} \label{section:standardkernels}
There is a number of related notions of "standard" Calder\'on-Zygmund kernels in the context of homogeneous spaces, see \cite[p. 29]{stein1}\cite{han}\cite{hofmann}.  We say that $K$ is a Calder\'on-Zygmund kernel  with respect to a quasi-metric $d$, corresponding quasi-balls $\buu$ and a measure $\leb$,  if  there exist constants $k>1$, $\epsilon>0$, and $C$ such that
for all $x,y,\bar y$ the following estimates hold:
\begin{equation} \label{czprop1}
\betrag{K(x,y)}\le \frac{C}{\leb(\buu_{d(x,y)}(y))} \quad\text{and}
\end{equation}
\begin{equation} \label{czprop2}\begin{split} 
\betrag{K(y,x)-K(\bar y,x)}+\betrag{K( x, y)-K(x,\bar y)} \\
\le C\left(\frac{d(\bar y, y)}{d(x,y)}\right)^\epsilon \cdot\frac{1}{\leb(\buu_{d(x,y)}(y))}
\end{split} \qquad 
\text{ if } \ \  k\cdot d(y,\bar y)<d(x,y). \end{equation}
There are hints that the kernels introduced in Sect. \ref{sec:kern} are of this type.
We examine an example in the setting made up in Sect. \ref{sec:example}.
\begin{theorem} Let $G=G_1$, $D$, $d$, and $O$ as in Sect. \ref{sec:example} and let $K$ be a kernel $K_\psi$, where $\psi(\ortp x)=\psi(x)$, $\supp\psi\subseteq B_1$ and $\int\psi=0$, as defined in Sect. \ref{sec:kern}. Then there exist $k$, $c$, $\epsilon$, and $C$ such that $K$ is a Calder\'on-Zygmund kernel in the sense of \eqref{czprop1} and \eqref{czprop2}.
\end{theorem}
\begin{proof} 
We choose some $k\ge 4$ such that under the condition $k\cdot d(y,\bar y)<d(x,y)$, we have 
 $d(x,y)\simeq d(x,\bar y)$ as in \eqref{distequiv} and furthermore $4d(y,\bar y)<d(x,\bar y)$.
Then it is sufficient to show \eqref{czprop1} and \eqref{czprop2} with $d(x,y)$ replaced by $\Delta=\min\{d(x,y), d(x,\bar y)\}$.  Note that if $x-\ortn y\in \supp\psi_t$ or $\ortp x-y\in \supp\psi_t $, then $\betrag{x-\ortn y}<t$ and consequently $d(x,y)<t$.
Thus we have
\begin{equation*}
K(x,y)=\int_\Delta^\infty \psi_t(x-\ortn y)\dht  
\end{equation*}
and  Lemma \ref{dichte} with $p=0$ yields
\begin{equation*}\begin{split}
\betrag{K(x,y)}&\le \betrag{\int_\Delta^\infty t^{-Q}\norm{\psi}_\infty \cdot 1_{\menge{t}{\betrag{x-\ortn y}<t}} \dht} \\
&\le C \norm{\psi}_\infty \frac{1}{\leb(\buu_\Delta(y))},
\end{split}
\end{equation*}
proving \eqref{czprop1}. We continue with the proof of \eqref{czprop2}, which is trivial if $y=\bar y$, so we assume $y\neq\bar y$.  Due to the rotational symmetry of $\psi$, we have
\begin{equation}\label{roty}
K(\bar y,x)=\int_\Delta^\infty \psi_t(\ortp \bar y-  x)\dht. 
\end {equation}
Setting $\delta=d(y,\bar y)$, the estimate $0<4\delta<\Delta$ holds, and it suffices to verify that 
\begin{equation} \label{cz2a}
\betrag{K( x, y)-K(x,\bar y)} 
\le C \left( \frac{\delta}{\Delta} \right)^\epsilon \cdot \frac{1}{\leb(\buu_{\Delta}(y))}
\end{equation}
and
\begin{equation} \label{cz2b}
\betrag{K(y,x)-K(\bar y,x)} 
\le C \left( \frac{\delta}{\Delta} \right)^\epsilon \cdot \frac{1}{\leb(\buu_{\Delta}(y))}. 
\end{equation}
Because of \eqref{roty}, \eqref{cz2b} can be shown with the same techniques as \eqref{cz2a}, so we prove \eqref{cz2a} only.
Choose some real number $s$, $\betrag{s}\le \delta$, such that $\betrag{\ort{s} y-\bar y}\le \delta$. 
Then $\betrag{s}\le\frac{1}{4}\Delta$ and
\[ \begin{split} 
&\betrag{K(x,y)-K(x,\bar y)}\\
&= \betrag{\int_{\Delta}^\infty t^{-Q-1}\left[\psi(D_t^{-1}(x-\ortn y )-\psi(D_t^{-1}(x-\ortn \bar y))\right]\diff t }\\
& = \betrag{\int_{\frac{3}{4}\Delta}^\infty  \left[ (t-s)^{-Q-1}\psi(D_{t-s}^{-1}(x- O_{-t+s}y) )-t^{-Q-1}\psi(D_t^{-1}(x-\ortn \bar y))\right]\diff t }\\
& \le \int_{\frac{3}{4}\Delta}^\infty  (t-s)^{-Q-1} \cdot\betrag{ \psi(D_{t-s}^{-1}(x-\ort{-t+s} y) )-\psi(D_{t-s}^{-1}(x-\ortn \bar y)) } \diff t \\
&\quad + \int_{\frac{3}{4}\Delta}^\infty  \betrag{ (t-s)^{-Q-1}- t^{-Q-1}} \cdot \betrag{ \psi(D_{t-s}^{-1}(x-\ortn \bar y) )} \diff t \\
&\quad+ \int_{\frac{3}{4}\Delta}^\infty   t^{-Q-1}\cdot\betrag{\psi(D_{t-s}^{-1}(x-\ortn \bar y) )- \psi(D_t^{-1}(x-\ortn \bar y)))}\diff t \\
&= \int_{\frac{3}{4}\Delta}^\infty f_1(t)\dht\ +\ \int_{\frac{3}{4}\Delta}^\infty f_2(t)\dht\ +\ \int_{\frac{3}{4}\Delta}^\infty f_3(t)\dht,
\end{split}\]
where $f_1$, $f_2$, and $f_3$ are the integrands in the preceding lines. They are supported in the sets

\[\begin{split} \supp f_1&\subseteq \menge{t>0}{ \betrag{x-\ortn \bar y }<2t\ \text{ or }\   \betrag{x-\ortn(\ort{s} y) }<2t} ,\\
\supp f_2,\ \supp f_3&\subseteq \menge{t>0}{\betrag{x-\ortn \bar y }<2t}. \end{split}\]
Lemma \ref{dichte} yields

 \begin{equation} \label{suppintab}\int_{\frac{3}{4}\Delta}^\infty t^{-Q-p} \cdot 1_{\{\supp f_j\}} \dht\ \lesssim \frac{\Delta^{-p}}{\mu(\buu_\Delta(x))}\lesssim \frac{\Delta^{-p}}{\mu(\buu_\Delta(y))} \kq j=1,2,3. \end{equation}
We now complete the proof by estimating the integrands $f_j$ against appropriate powers of $t$.
We have
\[ \begin{split} 
f_1(t)&=  (t-s)^{-Q} \cdot\betrag{  \psi(D_{t-s}^{-1}(x-\ort{-t+s} y) )-\psi(D_{t-s}^{-1}(x-\ortn \bar y)) } \\
&\lesssim \norm{\nabla \psi}_\infty t^{-Q} \norm{D_{t-s}^{-1}(\ort{s}y-\bar y)}_2  \\
&=\norm{\nabla \psi}_\infty t^{-Q} \norm{D_\frac{\delta}{t-s}D_\delta^{-1}(\ort{s}y-\bar y)}_2 \\
&\le\norm{\nabla \psi}_\infty  t^{-Q} \norm{D_\frac{\delta}{t-s}} \\
&\le C_1 \norm{\nabla \psi}_\infty \delta^\gamma t^{-Q-\gamma} ,
\end{split}\]
where $\gamma=\min\{ a,a_2,\ldots,a_n\}$.
Together with \eqref{suppintab} this implies the inequality
\begin{equation}\begin{split}\label{intf1}
&\int_{\frac{3}{4}\Delta}^\infty f_1(t)\dht\ \\
&\le C_1\norm{\nabla \psi}_\infty \delta^\gamma \int_{\frac{3}{4}\Delta}^\infty t^{-Q-\gamma}\cdot1_{\{\supp f_1\}}\dht  \\
&\le C_2\norm{\nabla \psi}_\infty \left(\frac{\delta}{\Delta}\right)^\gamma \frac{1}{\mu( \buu_\Delta(X))} .
\end{split}\end{equation}
In a similar way we obtain 
\[ \begin{split} 
f_2(t)&=\betrag{ (t-s)^{-Q-1}- t^{-Q-1}} \cdot \betrag{ \psi(D_{t-s}^{-1}(x-\ortn \bar y) )} t \\ 
& \le C \norm{\psi}_\infty \delta t^{-Q-1},\end{split}\]
\begin{equation}
\int_{\frac{3}{4}\Delta}^\infty f_2(t)\dht\ 
\le \norm{\nabla \psi}_\infty r \int_{\frac{3}{4}\Delta}^\infty t^{-Q-1}\cdot1_{\{\supp f_2\}}\dht\
 \le \norm{\psi}_\infty \frac{\delta}{\Delta} \cdot\frac{1}{\leb( \buu_\Delta(y))} .
\end{equation}
 Furthermore, we have the estimate
\[ \begin{split} 
f_3(t)&=t^{-Q}\cdot\betrag{\psi(D_{t-s}^{-1}(x-\ortn \bar y) )- \psi(D_t^{-1}(x-\ortn \bar y)} \\
&\le\norm{\nabla \psi}_\infty  t^{-Q}\cdot\norm{(D_{t-s}^{-1}-D_t^{-1})(x-\ortn \bar y) }_2 \\
&=\norm{\nabla \psi}_\infty \  t^{-Q}\cdot\norm{(D_{1-\frac{s}{t}}^{-1}-\id)D_t^{-1}(x-\ortn \bar y) }_2 \\
&\le\norm{\nabla \psi}_\infty  t^{-Q}\cdot\norm{(D_{1-\frac{s}{t}}^{-1}-\id)}.
\end{split}\]
Note that $p\mapsto D_{1-p}^{-1}-\id$ is smooth, mapping $0$ to $0$, and that $\betrag{\frac{s}{t}}\le\betrag{\frac{\delta}{t}}\le\frac{1}{3}$. So we find a constant such that
\[ 
f_3(t)\le C\norm{\nabla \psi}_\infty \delta  t^{-Q-1},
\]
and
\begin{equation}\label{intf3}
\int_{\frac{3}{4}\Delta}^\infty f_3(t)\dht\ 
\le \norm{\nabla \psi}_\infty \delta \int_{\frac{3}{4}\Delta}^\infty t^{-Q-1}\cdot1_{\{\supp f_3\}}\dht\ 
\le \norm{\nabla\psi}_\infty \frac{\delta}{\Delta} \cdot\frac{1}{\leb( \buu_\Delta(y))}.
\end{equation}
Adding  inequalities \eqref{intf1} - \eqref{intf3} and using $\frac{\delta}{\Delta}\le 1$, we obtain a constant C such that
\[\betrag{K(x,y)-K(x,\bar y)}\le C\cdot (\norm{\psi}_\infty+\norm{\nabla\psi}_\infty) \left(\frac{\delta}{\Delta}\right)^{\min\{1, \gamma\ \}} \cdot\frac{1}{\leb( \buu_\Delta(y))}.\]
\end{proof}
\begin{acknowledgements}
This article evolved from my diploma thesis \cite{blo}.
I would like to thank Prof. M\"uller (Kiel) and Prof. Ricci (Pisa) for their advice on the subject. The idea of choosing a homogeneous norm  adapted to the family of rotations as in \eqref{rotinvariance} is due to Prof. Ricci. 
\end{acknowledgements}
 
\end{document}